\documentclass[a4paper,12 pt]{amsart}
\usepackage{hyperref,latexsym}
\usepackage{enumerate}

\theoremstyle{plain}
\newtheorem{theorem}{Theorem}[section]

\newtheorem{corollary}[theorem]{Corollary}
\newtheorem{proposition}[theorem]{Proposition}
\theoremstyle{definition}
\newtheorem{definition}[theorem]{Definition}
\newtheorem{example}{Example}
\theoremstyle{remark}
\newtheorem{remark}{Remark}

\newcommand{\abs}[1]{\left|#1\right|}

\begin{document}

\title[Maximizing orbits for Convex billiards]
      {Maximizing orbits \\ for higher dimensional \\
      convex billiards}

\date{23 July 2008}
\author{Misha Bialy}
\address{Raymond and Beverly Sackler School of Mathematical Sciences, Tel Aviv University,
Israel} \email{bialy@post.tau.ac.il}

\subjclass[2000]{}

\begin{abstract}
The main result of this paper is, that for convex billiards in
higher dimensions, in contrast with 2D case, for every point on the
boundary and for every $n$ there always exist billiard trajectories
developing conjugate points at the $n$-th collision with the
boundary. We shall explain that this is a consequence of the
following variational property of the billiard orbits in higher
dimension. If a segment of an orbit is locally maximizing, then it
can not pass too close to the boundary. This fact follows from the
second variation formula for the Length functional. It turns out
that this formula behaves differently with respect to "longitudinal"
and "transversal" variations.

\end{abstract}

\maketitle

\section{Introduction and Background }
\label{sec:intro} This paper is motivated by a question by
J\"{u}rgen Moser about conjugate points for Birkhoff billiards in
dimensions higher than 2. The main result of this paper is that in
higher dimensions, in contrast with 2D case, for every point on the
boundary and for every $n$ there always exist billiard trajectories
developing conjugate points precisely at the $n$-th collision with
the boundary. This fact, as we shall see, is a consequence of the
following variational property of the billiard orbits in higher
dimension. There are no locally maximizing configurations, passing
too close to the boundary. This follows from the second variation
formula for the Length functional. It turns out that this formula
behaves differently with respect to "longitudinal" and "transversal"
variations. As an example of this calculation let me mention the
example of standard sphere where the only orbits which are locally
maximizing are that of the diameters.

A natural extension of Moser's question is, if there are variational
properties of  billiard configurations distinguishing standard
sphere among other bodies. This remains still open, though we
formulate a conjecture in this direction below.

Let me denote by $\Sigma$ a strictly convex hypersurface in the
Euclidean space $\mathbb{R}^d$, we shall assume throughout this
paper, that it is at least $C^2$-smooth with strictly positive
curvature at every point. Any billiard configuration determines a
sequence of collision points with the boundary, $ \{x_n
\}_{n\in\mathbb{Z}}, x_n\in\Sigma$. Billiard configurations are in
one to one correspondence with the critical points of the functional
$$ \Phi \{u_n\}=\sum_{n\in\mathbb{Z}}{L(u_n,u_{n+1})}, $$
where
$$L:\Sigma\times\Sigma\rightarrow\mathbb{R},L(x,y)=\abs
{x-y},$$ denotes the Euclidean distance function in $\mathbb{R}^d$.
More precisely, this correspondence means, that any finite segment
of trajectory $ \{x_n \}_{n\in [M,N]}$ is the critical point of the
function
$$\Phi_{M,N}(u_M,\ldots,u_N)=L(x_{M-1},u_M)+\sum_{i=M}^{N-1}
L(u_i,u_{i+1})+ L(u_N,x_{N+1}).$$
  Notice that  $L$ has singularities on the diagonal which complicate the
variational analysis of the critical points of the functional $\Phi$
and $\Phi_{M,N}$. In this paper we shall adopt the following:

\begin{definition}
\begin{enumerate}[(a)]

\item
The segment of billiard orbit $ \{x_n \}_{n\in [M,N]}$ is called
maximizing if is a local maximum for the function $\Phi_{M+1,N-1}$.
\item
An infinite orbit is called maximizing if any finite segment of
it is maximizing.
\item
 The sequence of tangent vectors $\xi_n\in T_{x_n}\Sigma$ at the
points of a billiard configuration $\{x_n\}$ is called Jacobi field
it appears as a variation field ,
$\xi_n=\frac{d}{d\varepsilon}\mid_{\varepsilon=0}(x(\varepsilon)_n)$
for a variation of the initial configuration.
\item
Two points $x_M,x_N$ of the billiard ball configuration are called
conjugate if there is a non zero Jacobi field vanishing at
$x_M,x_N$.
\end{enumerate}
\end{definition}
Let me remark that the last definition of conjugate points has very
clear meaning on the language of geometric optics. It means that for
the light ray starting at $x_M$ becomes focused at $x_N$ , so in
other words it is a "bright" point on the boundary.

For the plane convex billiards there are lot of maximizing billiard
configurations. They appear in the so called Aubry -Mather sets. For
example any billiard configuration tangent to a smooth convex
caustic is maximizing ( notice that by KAM type theorem of Lazutkin
(\cite{lazutkin}), there are infinitely many convex caustics near
the boundary, for any sufficiently smooth convex billiard table). On
the other hand, I have proved in (\cite {bialy1}) that only for
circles all configurations are maximizing, in other words for any
non circular billiard there always exist conjugate points. In
contrast with this in higher dimensions for any shape with no
exceptions there always exist conjugate points and in fact many of
them.
\begin {theorem}
\label{conjugate} Let $\Sigma\subset\mathbb{R}^d,d>2$ be any
$C^2$-smooth strictly convex hypersurface with positive curvature.
Then for any point $x\in\Sigma$ there always exist conjugate points
along infinitely many configurations starting at $x$.
\end {theorem}
\begin{theorem}
\label{maximizing} There exists a constant $C(\Sigma)$ such that for
any maximizing configuration $\{x_n\}$ the angle of reflections
$\varphi_n$ at any vertex $x_n$ can not be too small:
$$\varphi_n>C(\Sigma)$$
\end{theorem}

Let me come back to the question by J.Moser. Analyzing the proofs of
the Theorems \ref{maximizing} and, \ref{conjugate} below, one
naturally comes to the following conjectures. Conjecturally there
are no convex hypersurfaces different from spheres, with the
property that all billiard orbits are maximizing with respect to
"longitudinal" perturbations. A somewhat related version of this
conjecture is the following: there are no convex hypersurfaces
different from spheres such that every billiard trajectory has a
non-vanishing longitudinal Jacobi field.  Let me remark, that the
application of my previous ideas from (\cite{bialy1}) lead naturally
to some new integral geometric quantities measuring certain
"roundedness" for bodies of constant width , and interesting
inequalities of isoperimetric type. However, I can not claim for the
moment that these inequalities distinguish spheres among other
convex bodies. Let me mention also that, if the conjectures are true
they could be considered as a kind of "non-holonomic" version of a
theorem by R Sine, saying that the spheres are the only convex
hypersurfaces with the property that any billiard configuration lies
in a 2-plane (see \cite{tabachnikov} for the proof and discussions).
I hope to discuss these questions elsewhere.

Let me finish this introduction with the following remark. It is in
fact an old idea going back at least to Hedlund and Morse, and alive
till nowadays, to construct invariant sets of Hamiltonian systems by
variational methods. However it is always a problem to decide if
such a set has zero, positive or full measure in the phase space.
Let me give here two examples: for planar billiards the set of
maximizing orbits has positive measure, by the result of
\cite{lazutkin}, mentioned above. It was proved in \cite{bialy1},
that the set of maximizing orbits can not occupy the set of full
measure in the phase space. It is a very interesting question what
happens in this perspective for convex billiards of higher
dimension. What can be said about the set of maximizing orbits or
about the set of maximizing with respect to longitudinal
perturbations. Let me remark also, that it was proved by M.Berger
that there are no convex caustics in higher dimensional billiards
except for the ellipsoids (see {\cite{berger}).

In the next  two section we shall derive the computation of second
derivatives of the function $L$ and the second variation formulas.
This will be crucial ,in fact, for the proof of the main theorems.
Then we shall provide examples in Section 4, and prove the existence
of conjugate points in Section 5.
\section*{Acknowledgements}
It is my pleasure to thank Misha Sodin and Serezha Tabachnikov for
useful and stimulating discussions.
\section {Derivatives of the distance function $L$}
Let me introduce first the Billiard ball map. This is symplectic
diffeomorphism of the unit ball cotangent bundle of $\Sigma$ for
which $L$ serves as a generating function. Everywhere in this paper
we will identify the cotangent and tangent bundles of $\Sigma$ with
the help of Riemannian metric induced on $\Sigma$. As usual for any
point $x \in \Sigma$ and for a inward unit vector $z \in
T_x\mathbb{R}^d$ one can define a vector $v \in B_x \Sigma , \abs
v<1$ by the projection $\pi_x$ of $z$ along the normal $n_x$ so
that, $$v=z - \hat{v}n_x \text { where }
\hat{v}=\sqrt{1-v^2}=\sin\varphi,$$  and $\phi$ stands for the angle
between $z$ and $T_x \Sigma$. By means of these notations the
billiard ball map  $T:B^{*} \Sigma \rightarrow B^{*} \Sigma$ is
defined implicitly by the following
\begin{align}
\label{orbit} T:(x,v)\mapsto (y,w) \Leftrightarrow
\nonumber & \\
 L_1(x,y)=-\pi_x(\frac{y-x}{\abs{y-x}})=-v,\quad
L_2(x,y)=\pi_y(\frac{y-x}{\abs{y-x}})=w
\end{align}
Here $L_1,L_2$ denote the (co)vectors, which are differentials of
$L$ with respect to $x,y$ respectively. Notice that the formulas
(\ref {orbit}) mean precisely that the orbits of $T$ are in one to
one correspondence with the extremals of $\Phi$ which are precisely
the billiard configurations. In fact the calculation of the mixed
derivative below (the so called twist condition) imply that $T$ is a
genuine diffeomorphism. These formulas for the second partial
derivatives of $L$ can be obtained by a direct calculation.
\begin{proposition}  For  any two distinct points $x,y\in\Sigma$ the linear
operators $L_{11}, L_{12},L_{21} L_{22}$ act by the following
formulas:
\begin{enumerate}[(a)]
\medskip
\item $ L_{11}(x,y) : T_x\Sigma \rightarrow T_x\Sigma,\qquad
L_{11}(\xi)=(\xi-<v,\xi>v)/L-S(\xi)\hat{v}$,
\medskip
\item $L_{22}(x,y) : T_y\Sigma \rightarrow T_y\Sigma,\qquad
L_{22}(\eta)=(\eta-<w,\eta>w)/{L}-S(\eta)\hat{w}$,
\medskip
\item $ L_{12}(x,y) : T_y\Sigma \rightarrow T_x\Sigma,\qquad
L_{12}(\eta)=(-\pi_x(\eta)+<w,\eta>v)/{L}$
\medskip
\item $ L_{21}(x,y) : T_x\Sigma \rightarrow T_y\Sigma,\qquad
L_{21}(\xi)=(-\pi_y(\xi)+<v,\xi>w)/{L}$
\end{enumerate}
\medskip
for any $\xi\in T_x\Sigma,\eta\in T_y\Sigma$. Here $S$ denotes the
shape operator of $\Sigma$: $S(\xi)=-\nabla_{\xi }n_x$. Moreover the
operators $L_{12}$and $L_{21}$ are isomorphisms which are adjoint
one to the other.
\end{proposition}
\begin{remark}
The last property of $L_{12}, L_{21}$ being isomorphisms is the so
called twist condition, replacing the Legandre condition of calculus
of variations for continuous time.
\end{remark}
\begin {proof}
We shall derive (a) first and then (d), all other formulas are
analogous. For (a) $y$ is fixed and $x$ varies in the direction of
$\xi\in T_x\Sigma$. Denote by$\nabla, \tilde{\nabla}$ the standard
connections on $\Sigma$ and $\mathbb{R}^d$ respectively. We have:
\begin{align*}L_{11}(\xi)=\nabla_{\xi}L_1(x,y)=
\nabla_{\xi}(\frac{x-y}{L}-<\frac{x-y}{L},n_x>n_x)= \\
=\pi_x\tilde{\nabla_{\xi}}(\frac{x-y}{L}-<\frac{x-y}{L},n_x>n_x)=\\
=\pi_x(\frac{\xi}{L}-\frac{x-y}{L^2}<L_1,\xi>)-
<\frac{x-y}{L},n_x>\nabla_{\xi}n_x=\\
=\frac{\xi}{L}-\frac{<L_1,\xi>}{L}L_1+<L_1,n_x>S(\xi)=\\
=\frac{\xi}{L}-\frac{<v,\xi>}{L}v-S(\xi)\hat{v}
\end{align*}
In order to prove (d) let $\gamma(t)$ be any curve with
$\gamma(0)=x$ and $\dot{\gamma}(0)=\xi$. Then we have
\begin{align*}
L_{21}(\xi)=
\frac{d}{dt}\mid_{t=0}(\frac{y-x}{L}-\frac{<y-x,n_y>n_y}{L})=\\
=\pi_y(-\frac{\xi}{L}-\frac{y-x}{L^2}<L_1,\xi>)=\\
=-\pi_y(\frac{\xi}{L})-\frac{L_2}{L}<L_1,\xi>=
-\frac{\xi}{L}+<\frac{\xi}{L},n_y>n_y+\frac{w}{L}<v,\xi>
\end{align*}
\end{proof}

\section{Second variation formulas}
Let $\{x_n\}_{n\in\mathbb{Z}}$ be a billiard configuration. Pick
$M\leq N$ and the tangent vectors $ \xi_n\in T_{x_n}\Sigma,\quad
n\in [M,N]$ . With the help of the operators of second partials the
quadratic form of the second variation for the functional
$\Phi_{M,N}$ is the following:
\begin{align}
\delta^2\Phi_{MN}(\xi_M,\ldots,\xi_N)=
\sum_{n=M}^{N}<(L_{11}(x_n,x_{n+1})+L_{22}(x_{n-1},x_n))\xi_n,\xi_n>+
\\
\nonumber +2\sum_{n=M}^{N-1}<L_{12}(x_n,x_{n+1})\xi_{n+1},\xi_n>.
\end{align}
Let me apply now the formulas of the Proposition in a very special
case, the orbit of one reflection, $\{x_{-1},x_{0}, x_{1}\}$, i.e.
$M=0, N=0$. We have
\begin{align}
\delta^2\Phi_{00}(\xi_0)=(\xi_0^2-<v_0,\xi_0>^2)(\frac{1}{L(x_{-1},x_0)}+\frac{1}{L(x_{0},x_1)})
-2B(\xi_0,\xi_0)\hat{v}_0,
\end{align}
where $B(\xi_0,\xi_0)=<S(\xi_0),\xi_0>$ is the second fundamental
form, and as before $v_0=\pi_{x_0}((x_1-x_0)/L(x_0,x_1))$ and
$\hat{v}_0=\sqrt{1-v^2_0}=\sin\varphi_0$ is the sinus of the angle
of reflection. Notice that this formula gives different answers for
the vectors $\xi_0$ orthogonal or parallel to $v_0$, as follows:

\begin{equation}
\label{orthogonal}
 \delta^2\Phi_{00}(\xi)=\xi^2
(\frac{1}{L(x_{-1},x_0)}+\frac{1}{L(x_{0},x_1)})
-2B(\xi,\xi)\hat{v}_0, \quad \text{for } \xi \bot v_0,
\end{equation}
\begin{equation}
\label{parallel}
 \delta^2\Phi_{00}(\xi)=\xi^2\hat{v}_0^2
(\frac{1}{L(x_{-1},x_0)}+\frac{1}{L(x_{0},x_1)})
-2B(\xi,\xi)\hat{v}_0,  \quad \text{for }\xi \| v_0
\end{equation}
Using these formulas we obtain the following
\begin{theorem}
\label{three} There exists a constant $C(\Sigma)>0$, such that for
any piece of billiard trajectory $\{x_{-1}, x_0,x_1\}$ having angle
of reflection $\varphi$ at the point $x_0$ smaller than $C(\Sigma)$
the following property holds:
$$\delta^2\Phi_{00}(\xi)>0 \text{ for all }\xi \bot v_0$$ and
$$\delta^2\Phi_{00}(\xi)<0 \text{ for all }\xi \| v_0$$
\end{theorem}
\begin{proof}Follows immediately from the explicit
formulas above. Indeed for $\varphi$ small enough in
(\ref{orthogonal})
$(\frac{1}{L(x_{-1},x_0)}+\frac{1}{L(x_{0},x_1)})$ becomes large
while $B(\xi,\xi)\hat{v}_0$ tends to zero. This gives the first
inequality of the theorem. In order to prove the second, one needs
to represent the surface near the point $x_0$ by a graph of a convex
function and uses Taylor expansions near $x_0$ in (\ref{parallel}).
We omit the details, since they are the same as in a planar case.
\end{proof}

The different behavior of the second variation for transversal and
longitudinal perturbations in the last theorem was discussed on a
different language of \emph{fronts} by L.Bunimovich
(\cite{bunimovich}).

\section{Proof of the Theorem \ref{maximizing}, Examples}
 Theorem \ref{maximizing} is an immediate consequence of
 Theorem \ref{three}, because any segment of a locally maximizing orbit
 has to have second variation negative semi-definite.
\begin{corollary}
Let $x\in \Sigma$ be a point with a "small" second fundamental form:
$$k_{\xi}=B(\xi,\xi)<\frac{1}{D}, \text { for all }
\xi\in T_x\Sigma \text { with } \abs {\xi}=1,$$ then no maximizing
segment passes through $x$.
\end{corollary}
\begin{proof}
It follows from  the identity (\ref{orthogonal}) for $\abs{\xi}=1$
$$\delta^2\Phi_{00}(\xi)=
(\frac{1}{L(x_{-1},x_0)}+\frac{1}{L(x_{0},x_1)})
-2B(\xi,\xi)\hat{v}_0>\frac{2}{D}-\frac{2}{D}\sin{\varphi_0}>0.$$
\end{proof}
\begin{remark} This corollary can be regarded as analogous to the
construction of Riemannian metrics with  "big bumps", where no
minimal geodesics pass through the top of the bump, see
\cite{b-p},\cite{bangert} for details. Let me mention also that no
such construction can be invented for planar Birkhoff billiards
where through any point on the boundary pass infinitely many
maximizing orbits due to existence of caustics near the boundary.
\end{remark}
\begin{example}
Consider  for example Ellipsoid in $\mathbb{R}^3$
$$E=\{\frac{x_1^2}{a_1^2}+\frac{x_2^2}{a_2^2}+\frac{x_3^2}{a_3^2}=1\},\text{ for } a_1\leq a_2\leq a_3.$$
Consider the point $A=(a_1,0,0)$ on the shortest axes. Then the
principle curvatures of $A$ are
$\frac{a_1}{a_2^2},\frac{a_1}{a_3^2}$  and the diameter is $D=2a_3$.
Therefore, if $2a_1a_3<a^2_2$ , then there are no maximizing orbits
passing through $A$.
\end{example}
Our next example is somewhat opposite to the previous one. It shows
that for certain shapes there are in fact many infinite maximizing
orbits.
\begin{example}
Let $ \gamma$ be a planar smooth strictly (of positive curvature)
convex closed curve. Let $\Sigma$ be smooth hypersurface in
$\mathbb{R}^3$ containing $\gamma$ and symmetric with respect to the
plane containing $\gamma$. Then obviously any orbit of the planar
billiard inside $\gamma$ remains an orbit of the billiard inside
$\Sigma$. Let me denote the principle curvatures of $\Sigma$ at the
points of $\gamma$ by $k_1$-in the direction of $\gamma$ , and
$k_2$-in the orthogonal direction. Let $\{x_n\}_{n\in \mathbb{Z}}$
be an infinite Aubry-Mather orbit in the plane of $\gamma$. Then all
the angles of reflections $\varphi_n$ are bounded away from zero,
i.e. there is $C_1>0$, so that $\sin\varphi_n>C_1$. We claim, that
if $k_2$ is sufficiently large then the orbit $\{x_n\}_{n\in
\mathbb{Z}}$ remains maximizing inside the billiard in $\Sigma$.
\begin{proof}In order to prove the claim we shall examine the second
variation. Any field $\xi_n$ along the orbit $\{x_n\}$ can be spited
into a sum of two, the first is in the direction of $\gamma$ and the
second is in the orthogonal direction,
$\xi_n=\xi^{\prime}_n+\xi^{\prime\prime}_n $. Then we have for the
second variations $$ \delta^2\Phi(\xi)= \delta^2\Phi(\xi^{\prime})+
\delta^2\Phi(\xi^{\prime\prime})+2\delta^2\Phi(\xi^{\prime},\xi^{\prime\prime})$$
where the mixed term
$2\delta^2\Phi(\xi^{\prime},\xi^{\prime\prime})$ vanishes, due to
the symmetry assumptions. Since the orbit $\{x_n\}_{n\in
\mathbb{Z}}$ is the Aubry-Mather orbit inside $\gamma$ then it is
maximizing with respect to planar perturbations and therefore
$\delta^2\Phi(\xi^{\prime})<0$. Let us estimate the
$\delta^2\Phi(\xi^{\prime\prime})$. Denote by
$$K_1=\max_{\gamma}k_1,\qquad K_2=\min_{\gamma}k_2.$$ For the planar
billiard one has, $$
L(x_n,x_{n+1})>2\sin\varphi_n/K_1>2C_1/K_1,\text{ for all
}n\in\mathbb{Z}.$$ Let $a_n, b_n$ be the diagonal and the
off-diagonal elements of the quadratic form
$\delta^2\Phi(\xi^{\prime\prime})$ respectively. Then $$
a_n=(\frac{1}{L(x_{n-1},x_n)}+\frac{1}{L(x_{n},x_{n+1})})
-2k_{2}(x_n)\sin\varphi_n <K_1/C_1-2K_2:=a$$ Also $$ \abs
{b_n}=1/L(x_n, x_{n+1})<K_1/2C_1:=b.$$ Let the number $K_2$ be big
enough so that $a<0 \text { and } -a>2b$. Taking into account all
the estimations we obtain:
$$\delta^2\Phi_{MN}(\xi^{(2)})=\sum_{n=M}^N a_n{\xi^{\prime\prime}_n}^2+2\sum_{n=M}^{N-1}
b_n\xi^{\prime\prime}_n \xi^{\prime\prime}_{n+1}<a\sum_{n=M}^N
{\xi^{\prime\prime}_n}^2+2b\sum_{n=M}^{N-1}\xi^{\prime\prime}_n
\xi^{\prime\prime}_{n+1}
$$
It is an easy exercise to see that the last expression is negative.
This completes the proof.
\end{proof}
\end{example}
\begin{example} This is example of the standard sphere. In this
example non of the orbits except diameters are maximizing. For
instance all periodic orbits are not maximizing. Let me remark that
they are maximizing for the variational principle on the set of
closed n-gons. This shows the difference  between the two
variational principle . We refer the reader to recent paper
(\cite{farber}) for the results on periodic trajectories in general
and estimation of there number in higher dimension. Any orbit of the
billiard inside the standard sphere is completely determined by the
angle of reflection $\alpha$ (for sphere they are all the same).
Computing the quadratic form of the second variation with respect to
transversal perturbations gives the following matrix$$
\begin{pmatrix}
a & b & 0 & \cdots & 0 & 0 \\[1mm]
b & a & b & 0 & \cdots & 0 \\[1mm]
\vdots & \vdots & \vdots & \vdots & \vdots & \vdots \\[1mm]
0 &
0 & 0& \cdots b & a & b \\
0 & 0 & 0 & \cdots & b & a
\end{pmatrix},
\text{ where  } a=\frac{\cos2\alpha}{\sin\alpha} ,\quad
b=-\frac{1}{2\sin\alpha}
$$
It is an easy exercise to show that for any angle $\alpha$ choosing
the size of this matrix is big enough one gets the matrix which is
not negative definite. Let me remark here that the quadratic form
with respect to the longitudinal perturbations gives the same matrix
with the parameters $a=-\sin\alpha, b=\sin\alpha/2$, which is
negative definite.
\end{example}

\section{Conjugate points, Proof of Theorem \ref{conjugate}}
Let me introduce some notations. For a fixed point $x\in\Sigma$ and
$n\geq 1$ introduce the subset $M_{x,n}$ of $B^*_x\Sigma$ consisting
of those (co)vectors $v$ such that the corresponding billiard
trajectory $\{x_0=x, x_1,\ldots,x_{n+1}\}$ is maximizing , that is
$\delta^2\Phi_{1,n} $ is negative semi-definite. Obviously $M_{x,n}$
is not empty since given $x,y\in\Sigma$ one can set
$x_0=x,x_{n+1}=y$ and find $\{x_1,\ldots,x_n\}$
 giving the global maximum to the length functional $\Phi_{1,n}$.
 Moreover $M_{x,n}$ is closed by the continuity of the quadratic form
 $\delta^2\Phi_{1,n} $. It follows from Theorem \ref{maximizing}
 that this set is compact in the open ball $B^*_x\Sigma$, because no
 maximizing orbits near the boundary of the ball are allowed.
 Now it is easy to prove Theorem \ref{conjugate}
 \begin{proof}
Let $v$ be any vector lying in the boundary $\partial M_{x,n}$ and
consider the orbit corresponding to $(x,v)$: $\{x_0=x,
x_1,\ldots,x_{n+1}\}$. Then it follows that the quadratic form of
this orbit $\delta^2\Phi_{1,n} $ is negative semi-definite and also
has a nontrivial Kernel. Then there exists  a field $\{\xi_1,
\ldots,\xi_n\}$ lying in the Kernel. Then it must satisfy the
following equation:
\begin{align*}
(L_{22}(x_{k-1},x_k)+L_{11}(x_k,x_{k+1}))\xi_k+             &\\
L_{21}(x_{k-1},x_k)\xi_{k-1}+ L_{12}(x_{k},x_{k+1})\xi_{k+1}=0,
\text { for all  } k=1,\ldots,n
\end{align*}
This is precisely the equation of the Jacobi fields.
\newline
Thus $\{\xi_0=0,\xi_1, \ldots,\xi_n,\xi_{n+1}=0\}$ is a Jacobi field
vanishing at the ends. Therefore, $x_0, x_{n+1}$ are conjugate. This
yields the proof.
\end{proof}
Notice that $M_{x,n+1} \subseteq M_{x,n}$ and by the compactness it
follows that their intersection is not empty $\bigcap
M_{x,n}\neq\emptyset$.
\begin{corollary} For any point $x\in\Sigma$ there exists
infinite semi-orbit $\{x_n \}_{n\geq 0}, x_0=x$ starting at $x$
which is a maximizing.
\end{corollary}
Let me remark that the semi-orbit in the last  Corollary is not
claimed to be an infinite orbit , see the Example 1 where no
infinite orbit passes through a  certain point. It would be
interesting to know any example different from bodies of constant
width, where there are maximizing orbits passing through every point
of $\Sigma$.
\begin{remark}
Let me explain, that the compacts $M_{x,n}$ are in fact rather
"fat". Denote by  $M^\prime_{x,n}$  the set of all those $v\in
B^*_x\Sigma$ such that the corresponding orbit $\{x_0=x,
x_1,\ldots,x_{n+1}\}$ has negative definite quadratic form
$\delta^2\Phi_{1,n}$ . Obviously $M^\prime_{x,n}\subseteq M_{x,n}$.
Moreover the following important inclusion holds:
$$M^\prime_{x,n+1}\subseteq M_{x,n},$$
which means that any proper subsegment of a maximizing segment has a
non-degenerate second variation form. This is of course a well known
fact in Riemannian case. It was proved  for twist maps in
(\cite{mms}) ( in the case of higher dimensional billiards the twist
condition is that the operators $L_{12} \text { and } L_{21}$ are
isomorphisms and the  proof of (\cite{mms}) goes through with no
change). We refer also to (\cite{bm}) for more discussions on the
twist maps.

\end{remark}



\end{document}